\documentclass[12pt]{amsart}

\usepackage{amsthm}
\usepackage{amssymb}
\usepackage{cite}
\usepackage{enumerate} 
\usepackage{fullpage}
\usepackage[utf8]{inputenc}
\usepackage{hyperref}

\usepackage[capitalise]{cleveref}

\author{Maur\'icio Collares Neto}
\author{Robert Morris}

\title{Maximum-size antichains in random set-systems}
\address{IMPA, Estrada Dona Castorina 110, Jardim Bot\^anico, Rio de Janeiro, RJ, Brasil}
\email{\{collares|rob\}@impa.br}
\thanks{Research supported in part by a CAPES bolsa Proex (MCN) and by CNPq Proc.~479032/2012-2 and Proc.~303275/2013-8 (RM)}

\newtheorem{thm}{Theorem}[section]
\newtheorem*{HCL}{The Hypergraph Container Lemma}

\newtheorem{lemma}[thm]{Lemma}

\newtheorem{cor}[thm]{Corollary}

\theoremstyle{definition}

\newtheorem{defn}[thm]{Definition}

\newcommand{\binomn}{\binom{n}{n/2}}

\renewcommand{\c}{\operatorname{c}}
\newcommand{\C}{\operatorname{\mathcal{C}}}

\newcommand{\B}{\operatorname{\mathcal{B}}}

\def\le{\leqslant}
\def\leq{\leqslant}
\def\ge{\geqslant}
\def\geq{\geqslant}
\def\A{\mathcal{A}}
\def\C{\mathcal{C}}
\def\F{\mathcal{F}}
\def\G{\mathcal{G}}
\def\HH{\mathcal{H}}
\def\I{\mathcal{I}}
\def\P{\mathcal{P}}
\def\N{\mathbb{N}}
\def\eps{\varepsilon}

\begin{document}
  \begin{abstract}
    We show that, for $pn \to \infty$, the largest set in a $p$-random sub-family of the power set of $\{1, \ldots, n\}$ containing no $k$-chain has size $( k - 1 + o(1) ) p \binom{n}{n/2}$ with high probability. This confirms a conjecture of Osthus.
  \end{abstract}

  \maketitle

  \section{Introduction}

  One of the cornerstones of extremal set theory is the famous theorem of Sperner~\cite{MR1544925}, who proved in 1928 that the largest antichain in $\mathcal{P}(n)$, the family of all subsets of $\{1, \ldots, n\}$, has size $\binom{n}{n/2}$. In 1945, Erd\H{o}s~\cite{Erdos} generalized this result by showing that any family of sets larger than the $k-1$ middle layers of $\mathcal{P}(n)$ contains a $k$-chain.

  In this paper we will prove a sparse random analogue of Erd\H{o}s' theorem. More precisely, for every function $p \gg 1/n$ we will determine, with high probability, the (asymptotic) size of the largest sub-family of $\mathcal{P}(n,p)$, the $p$-random sub-family\footnote{That is, $\mathcal{P}(n, p)$ is a random variable such that $\mathbb{P}(A \in \mathcal{P}(n, p)) = p$ for each $A \in \mathcal{P}(n)$, and such events are independent for different values of $A$.} of $\mathcal{P}(n)$, containing no $k$-chain. This confirms a conjecture of Osthus~\cite{MR1757282}.

    \begin{thm}\label{thm:main}
     Let $2 \le k \in \mathbb{N}$, let $p = p(n)$ be such that $pn \to \infty$. Then the largest family $\A \subset \mathcal{P}(n,p)$ containing no $k$-chain has size
      \begin{equation}\label{eq:mainthm}
        |\A| = \big( k - 1 + o(1) \big) p \binomn
      \end{equation}
      with high probability as $n \to \infty$.
   \end{thm}

We remark that the bound on $p$ is best possible, since the result fails to hold whenever $pn \to C$. Indeed, in this case Osthus~\cite{MR1757282} showed that, with high probability, the two middle layers of $\mathcal{P}(n, p)$ contain an antichain $\mathcal{A}$ of size $\big( 1 + e^{-C/2}+o(1) \big)p\binomn$; adding $k-2$ further layers to $\mathcal{A}$ gives a family of size $\big( k - 1 + e^{-C/2}+o(1) \big)p\binomn$ containing no $k$-chains.

The alert reader may have noticed that in order to obtain Theorem~\ref{thm:main} it is sufficient to prove it in the case $k = 2$, since every family $\A \subset \mathcal{P}(n)$ containing no $k$-chain can be decomposed into $k-1$ antichains. (We thank Oliver Riordan for pointing this out to us.) Our proof will not use this simple fact, however,  proceeding instead via an application of the hypergraph container method (see Section~\ref{sec:containers}) to the $k$-uniform hypergraph encoding the $k$-chains of $\P(n)$. While more complicated that strictly necessary, this approach has two significant advantages: it motivates the proof of Theorem~\ref{thm:supersaturation}, below, which we consider to be of independent interest, and it gives an easily accessible introduction to the approach\footnote{This approach allows one to efficiently apply the hypergraph container lemma an unbounded number of times by combining it with a suitable `balanced supersaturation theorem', such as Theorem~\ref{thm:supersaturation}. It was used in~\cite{2013arXiv1309.2927M}, to prove (amongst other things) that the number of $C_{2k}$-free graphs with $n$ vertices is $2^{O(n^{1 + 1/k})}$.} of~\cite{2013arXiv1309.2927M}, which we expect to have numerous other applications.

  The study of the random set-system $\mathcal{P}(n, p)$ was initiated in 1961 by R\'enyi~\cite{MR0144366}, who determined the threshold for the event that $\mathcal{P}(n, p)$ is an antichain. More recently, Kreuter~\cite{K} and Kohayakawa, Kreuter and Osthus~\cite{KKO} studied the length of the longest chain in $\P(n,p)$, and Kohayakawa and Kreuter~\cite{KK} and Osthus~\cite{MR1757282} studied the size of the largest antichain. In particular, Osthus~\cite{MR1757282} proved that~\eqref{eq:mainthm} holds if $pn \gg \log n$, and conjectured that $pn \gg 1$ is sufficient. We note that this conjecture has also been proved independently by Balogh, Mycroft and Treglown~\cite{BMT}, who also studied sparser random set-systems.

 The problem of obtaining sparse random analogues of classical results in extremal combinatorics has attracted a large amount of attention in recent years, culminating in the extraordinary breakthroughs of Conlon and Gowers~\cite{2010arXiv1011.4310C} and Schacht~\cite{Schacht09}, who developed general techniques for solving such problems, and as a result were able to prove a number of longstanding conjectures, such as sparse random analogues of the theorems of Tur\'an and Szemer\'edi. A third approach, now known as the `hypergraph container method', was subsequently developed independently by Balogh, Morris and Samotij~\cite{2012arXiv1204.6530B} and by Saxton and Thomason~\cite{2012arXiv1204.6595S}. We will use this latter method in order to prove Theorem~\ref{thm:main}.

 In order to effectively apply the hypergraph container method (see Section~\ref{sec:containers}), one requires a so-called `balanced supersaturation theorem', and the proof of such a result (see Theorem~\ref{thm:supersaturation}, below) is the main innovation of this paper. An `unbalanced' supersaturation theorem (giving a lower bound on the number of $k$-chains, but not controlling the distribution of these chains) was proved by Kleitman~\cite{MR0232685} in the case $k = 2$, and by Das, Gan and Sudakov~\cite{2013arXiv1302.5210D} in general. More precisely, the authors of~\cite{2013arXiv1302.5210D} used the permutation method pioneered by Katona and LYMB\footnote{The acronym LYMB refers to Lubell~\cite{MR0194348}, Yamamoto~\cite{MR0067086}, Me\v{s}alkin~\cite{MR0150049} and Bollob\'as~\cite{MR0183653}. It often causes spelling confusion due to the silent B.} in order to show that a family with $t$ more elements than the extremal example above contains $\Omega\big( tn^{k-1} \big)$ $k$-chains. One of the key ideas from~\cite{2013arXiv1302.5210D} will also play an important role in our proof, see Lemma~\ref{lemma:dgs} below.

 In order to state our balanced supersaturation theorem, we will need a couple of simple definitions. For each $k \ge 2$ and $n \in \N$, let $\G_k = \G_k(n)$ denote the $k$-uniform hypergraph on vertex set $\P(n)$ whose edges encode $k$-chains, i.e., $\{ F_1, \ldots, F_k \} \in E(\G_k)$ if and only if $F_1 \supsetneq \cdots \supsetneq F_k$ for some ordering of the elements. Given $\mathcal{F} \subset \mathcal{P}(n)$, we write $\mathcal{H} \subset \G_k[\mathcal{F}]$ to denote that $\HH$ is a $k$-uniform hypergraph with vertex set $\F$ whose edges are all members of $E(\G_k)$. For each $\ell \in [k]$, we write $\Delta_\ell(\mathcal{H})$ for the maximum degree of a set of $\ell$ vertices of $\HH$, that is
 $$\Delta_\ell(\mathcal{H}) = \max\big\{ d_\HH(\A) \,:\, \A \subset V(\HH), \, |\A| = \ell \big\},$$
where $d_\HH(\A) = \big| \big\{ \B \in E(\HH) : \A \subset \B \big\} \big|$. We also write $\I(\HH)$ for the collection of independent sets of $\HH$, and $\alpha(\HH)$ for the size of the largest member of $\I(\HH)$.

 We can now state the key new tool that we will use to prove Theorem~\ref{thm:main}. It says that a family with slightly more than $\alpha(\G_k) = \big( k - 1 + o(1) \big) \binom{n}{n/2}$ elements not only contains many $k$-chains, but that these chains can be chosen to be fairly `evenly distributed' over $\P(n)$.

  \begin{thm}\label{thm:supersaturation}
    For every $k \ge 2$ and $\alpha > 0$, there exists $\delta = \delta(\alpha, k) > 0$ such that the following holds. Let $n \in \mathbb{N}$ and $\mathcal{F} \subset \mathcal{P}(n)$ satisfy $|\mathcal{F}| \ge (k-1+\alpha)\binom{n}{n/2}$, and suppose that $\delta^{-1} \le m \le \binom{|F|}{|G|}$ for every $F,G \in \F$ with  $F \supsetneq G$. Then there exists $\mathcal{H} \subset \G_k[\mathcal{F}]$ satisfying
    \begin{enumerate}[(a)]
      \item[$(a)$] $e(\mathcal{H}) \ge \delta^k m^{k-1} \binom{n}{n/2}$,
      \item[$(b)$] $\Delta_\ell(\mathcal{H}) \le (\delta m)^{k-\ell}$ \quad for every $1 \le \ell \le k$.
    \end{enumerate}
  \end{thm}

We remark that the bounds in Theorem~\ref{thm:supersaturation} are all close to best possible. To see this, set $m = n/3$ and consider the $k-1$ middle layers of the hypercube, together with $\alpha {n \choose n/2}$ elements from the next layer up. Then $\G_k[\F]$ has $O\big( n^{k-1} \binom{n}{n/2} \big)$ edges and $\Delta_\ell\big( \G_k[\F] \big) = \Omega(n^{k-\ell})$ for every $1 \le \ell \le k$. The technical assumption $m \le \binom{|F|}{|G|}$ for every $F,G \in \F$ with $F \supsetneq G$ will be useful because it will allow us to deduce sufficiently strong bounds both when $|\F|$ is close to $\alpha(\G_k)$, and when it is much larger, see Section~\ref{sec:containers}.

   The rest of this paper is organised as follows. In Section~\ref{sec:containers} we outline the hypergraph container method, in Section~\ref{sec:supersaturation} we prove Theorem~\ref{thm:supersaturation}, and in Section~\ref{sec:putting_it_together} we perform the necessary technical computations in order to deduce Theorem~\ref{thm:main}.

  \section{Hypergraph containers}\label{sec:containers}

  In this section, we will recall the powerful method of hypergraph containers, which was recently introduced in~\cite{2012arXiv1204.6530B,2012arXiv1204.6595S}. Roughly speaking, the method allows us to deduce from a balanced supersaturation theorem (such as Theorem~\ref{thm:supersaturation}) that there exists a relatively small family of `containers', each not too large, which cover the family $\mathcal{I}(\mathcal{H})$ of independent sets of a $k$-uniform hypergraph $\mathcal{H}$. The key container lemma (see~\cite[Proposition~3.1]{2012arXiv1204.6530B} and~\cite[Theorem~3.4]{2012arXiv1204.6595S}) is as follows.

  \begin{HCL}
    For every $k \in \mathbb{N}$ and $c > 0$, there exists a $\delta > 0$ such that the following holds. Let $\tau \in (0,1)$ and suppose that $\mathcal{H}$ is a $k$-uniform hypergraph on $N$ vertices such that
    \begin{equation}\label{eq:containers:condition}
      \Delta_\ell(\mathcal{H}) \le c \cdot \tau^{\ell-1} \frac{e(\mathcal{H})}{N}
    \end{equation}
 for every $1 \le \ell \le k$. Then there exist a family $\mathcal{C}$ of subsets of $V(\mathcal{H})$, and a function $f \colon \P\big( V(\mathcal{H}) \big) \to \mathcal{C}$ such that:
    \begin{enumerate}
      \item[$(a)$] For every $I \in \mathcal{I}(\mathcal{H})$ there exists $T \subset I$ with $|T| \le k \cdot \tau N$ and $I \subset f(T) \cup T$,\smallskip
      \item[$(b)$] $|C| \le (1 - \delta) N$ for every $C \in \mathcal{C}$.
    \end{enumerate}
  \end{HCL}

    We first apply this lemma to the hypergraph $\G_k$, to obtain a large family $\C_1$ of containers, each of size at most $(1 - \delta)2^n$. We then apply the lemma again, for each $\F \in \C_1$ with $|\mathcal{F}| \ge (k-1+\alpha)\binom{n}{n/2}$ (for some small $\alpha > 0$), to the hypergraph $\mathcal{H} \subset \G_k[\mathcal{F}]$ given by Theorem~\ref{thm:supersaturation}. We repeat this process until all containers have size at most $(k-1+\alpha)\binom{n}{n/2}$. The conditions~$(a)$ and~$(b)$ in Theorem~\ref{thm:supersaturation} allow us to check that~\eqref{eq:containers:condition} holds for a suitable value of $\tau$, and hence to count the containers in our final collection. See~\cite{2013arXiv1309.2927M} for a similar application of the container lemma in the context of $C_{2k}$-free graphs.

 In order to further motivate the statement of Theorem~\ref{thm:supersaturation} (and the technical condition $m \le \binom{|F|}{|G|}$ for every $F,G \in \F$ with $F \supsetneq G$), we will next deduce from it the following two lemmas, which we will use to check the condition~\eqref{eq:containers:condition}. The first shows that we can take $\tau = 1/n$ when $\F$ is slightly larger than $\alpha(\G_k)$.

 \begin{lemma}\label{cor:smallF}
 For every $k \ge 2$ and $\alpha > 0$, there exists $c = c(\alpha, k) > 0$ such that the following holds. Let $n \in \mathbb{N}$ be sufficiently large and $\mathcal{F} \subset \mathcal{P}(n)$ satisfy $(k-1+\alpha)\binom{n}{n/2} \le |\mathcal{F}| \le 3k \binom{n}{n/2}$. Then there exists $\mathcal{H} \subset \G_k[\mathcal{F}]$ satisfying
   $$\Delta_\ell(\mathcal{H}) \le \frac{c}{n^{\ell - 1}} \cdot \frac{e(\mathcal{H})}{|\F|}$$
   for every $1 \le \ell \le k$.
 \end{lemma}

 \begin{proof}
 First, observe that (by adjusting $\alpha$ slightly) we may assume that $|F| \ge n/3$ for every $F \in \F$, since the number of sets smaller than this is much smaller than ${n \choose n/2}$. Thus, applying Theorem~\ref{thm:supersaturation} with $m = n/3$, it follows that there exists a hypergraph $\mathcal{H} \subset \G_k[\mathcal{F}]$ and a constant $\delta = \delta(\alpha,k) > 0$ with $e(\mathcal{H}) \ge \delta^k m^{k-1} \binom{n}{n/2}$ and $\Delta_\ell(\mathcal{H}) \le (\delta m)^{k-\ell}$ for every $1 \le \ell \le k$. It follows that
   $$\Delta_\ell(\mathcal{H}) \le (\delta m)^{k-\ell} = \frac{3^\ell k}{\delta^{\ell} n^{\ell - 1}} \cdot \frac{\delta^{k} m^{k-1} \binom{n}{n/2}}{3k \binom{n}{n/2}} \le \frac{c}{n^{\ell - 1}} \cdot \frac{e(\mathcal{H})}{|\F|},$$
where $c = k \cdot (3/\delta)^k$, as required.
 \end{proof}

       The next lemma shows that if $|\F|$ is larger, then we can in fact take $\tau$ much smaller.

 \begin{lemma}\label{cor:bigF}
 For every $k \ge 2$, there exists $c = c(k) > 0$ such that the following holds. Let $n \in \mathbb{N}$ be sufficiently large and $\mathcal{F} \subset \mathcal{P}(n)$ satisfy $|\mathcal{F}| \ge 3k \binom{n}{n/2}$. Then there exists $\mathcal{H} \subset \G_k[\mathcal{F}]$ satisfying
   $$\Delta_\ell(\mathcal{H}) \le \frac{c}{n^{3\ell - 3}} \cdot \frac{e(\mathcal{H})}{|\F|}$$
   for every $1 \le \ell \le k$.
 \end{lemma}

 \begin{proof}
 First, choose an arbitrary partition $\F = \F_0 \cup \F_1 \cup \cdots \cup \F_t$ such that $|\F_i| = 3k \binom{n}{n/2}$ for every $i \in [t]$ and $|\F_0| < 3k\binom{n}{n/2}$. Fix $i \in [t]$, and observe that, by the pigeonhole principle, there are at least $k \binom{n}{n/2}$ elements of $\F_i$ whose sizes have the same remainder modulo $3$. Let $\F_i'$ be a collection of $\big( k - o(1) \big) \binom{n}{n/2}$ such elements, all of size at least $n/3$, and note that $\binom{|F|}{|G|} \ge \binom{n/3}{3}$ for every $F,G \in \F_i'$ with  $F \supsetneq G$. Thus, applying Theorem~\ref{thm:supersaturation} with $m = \binom{n/3}{3}$, it follows that there exists a hypergraph $\mathcal{H}_i \subset \G_k[\F_i']$ and a constant $\delta = \delta(k) > 0$ such that $e(\HH_i) = \big\lceil \delta^{k} m^{k-1} \binom{n}{n/2} \big\rceil$ and
   $$\Delta_\ell(\mathcal{H}_i) \le (\delta m)^{k-\ell} \le \frac{c'}{n^{3\ell - 3}} \cdot \frac{\delta^{k} m^{k-1} \binom{n}{n/2}}{k \binom{n}{n/2}} \le \frac{c'}{n^{3\ell - 3}} \cdot \frac{e(\mathcal{H}_i)}{|\F_i'|} $$
for some $c' = c'(k)$ and every $1 \le \ell \le k$. Setting $\HH = \HH_1 \cup \cdots \cup \HH_t$, it follows that
   $$\Delta_\ell(\mathcal{H}) \,\le\, \max_{1 \le i \le t} \big\{ \Delta_\ell(\mathcal{H}_i) \big\} \, \le \, \frac{c'}{n^{3\ell - 3}} \cdot \frac{\max_i e(\mathcal{H}_i)}{ \min_i |\F_i'|} \,\le\, \frac{c}{n^{3\ell - 3}} \cdot \frac{e(\mathcal{H})}{|\F|}$$
as claimed, since $e(\HH) = \sum_{i=1}^t e(\HH_i) = t \cdot e(\HH_i)$ and $|\F| \le 7t \cdot |\F'_i|$ for every $i \in [t]$.
 \end{proof}

  Motivated by the above bounds, fix $\tau \colon \mathcal{P}(n) \to \mathbb{R}$ to be the function defined by
      \begin{equation}\label{def:tau}
        \tau(A) := \begin{cases}n^{-1} &\text{if }|A| \le 3k \binom{n}{n/2}\\
                                n^{-3} &\text{otherwise.}\end{cases}
      \end{equation}
We can now specialize the Hypergraph Container Lemma to our application by combining it with Lemma~\ref{cor:smallF} and Lemma~\ref{cor:bigF}. The following corollary will be used in Section~\ref{sec:putting_it_together} to count the containers of a given size produced by repeated applications of the Hypergraph Container Lemma, see Theorem~\ref{thm:container_application}.

  \begin{cor}\label{cor:specialized_hcl}
    For every $2 \leq k \in \mathbb{N}$ and $\alpha > 0$, there exists $\delta = \delta(\alpha,k) > 0$ such that the following holds. Let $n \in \mathbb{N}$ be sufficiently large and $C \subset \mathcal{P}(n)$ with $|C| \ge (k-1+\alpha)\binomn$. Then there exists a collection $\mathcal{C} \subset \mathcal{P}(C)$ and a function $f \colon \mathcal{P}(C) \to \mathcal{C}$ such that
     \begin{enumerate}
        \item[$(a)$] For every $I \in \mathcal{I}(\mathcal{G}_k[C])$, there exists $T$ with $|T| \le k \cdot \tau(C) |C|$ and $T \subset I \subset f(T) \cup T$.\smallskip
        \item[$(b)$] $|C'| \le (1-\delta)|C|$ for every $C' \in \mathcal{C}$.
     \end{enumerate}
   \end{cor}

   \begin{proof}
   Apply the Hypergraph Container Lemma to the hypergraph $\HH \subset \G_k[C]$ given by Lemma~\ref{cor:smallF} (if $|C| \le 3k \binom{n}{n/2}$), or by Lemma~\ref{cor:bigF} (otherwise), and observe that (for a suitable choice of the constant $c$) the inequality~\eqref{eq:containers:condition} holds with $\tau = \tau(C)$ for every $1 \le \ell \le k$. It follows immediately that there exist a family $\mathcal{C}$ of subsets of $C$, and a function $f \colon \P(C) \to \mathcal{C}$ such that~$(a)$ and~$(b)$ hold, as required.
   \end{proof}

  \section{Balanced supersaturation}\label{sec:supersaturation}

 In this section, we will prove Theorem~\ref{thm:supersaturation} by constructing $\mathcal{H}$ one edge at a time. More precisely, starting with $\mathcal{H} = \emptyset$, we will repeatedly apply the following lemma, adding new edges to $\mathcal{H}$ until the conditions of Theorem~\ref{thm:supersaturation} are satisfied.

  \begin{lemma}\label{lemma:supersaturation}
    For every $k \ge 2$ and $\alpha > 0$, there exists $\delta = \delta(\alpha, k) > 0$ such that the following holds. Let $n \in \mathbb{N}$ and $\mathcal{F} \subset \mathcal{P}(n)$ satisfy $|\mathcal{F}| \ge (k-1+\alpha)\binom{n}{n/2}$, and suppose that $\delta^{-1} \le m \le \binom{|F|}{|G|}$ for every $F,G \in \F$ with  $F \supsetneq G$. If $\mathcal{H} \subset \G_k[\mathcal{F}]$ is a hypergraph satisfying
    \begin{enumerate}
      \item[$(a)$] $e(\mathcal{H}) \le \delta^k m^{k-1} \binom{n}{n/2}$,
      \item[$(b)$] $\Delta_\ell(\mathcal{H}) \le (\delta m)^{k - \ell}$ \quad for every $\ell \in [k]$,
    \end{enumerate}
    then there exists an edge $f \in \G_k[\mathcal{F}] \setminus \mathcal{H}$ for which $\Delta_\ell(\{f\} \cup \mathcal{H}) \le (\delta m)^{k - \ell}$ for every $\ell \in [k]$.
  \end{lemma}

  We remark that this approach to proving balanced supersaturation theorems -- adding the edges of $\HH$ one by one -- was also used in~\cite{2013arXiv1309.2927M}, and is likely to have further applications.

  The rest of this section will be dedicated to proving the Lemma~\ref{lemma:supersaturation}, so from now on let us fix $\alpha > 0$ and $k \ge 2$, and choose $\delta > 0$ sufficiently small and $m \ge \delta^{-1}$. Moreover, let us fix  $n \in \mathbb{N}$, a family $\mathcal{F} \subset \mathcal{P}(n)$ and a hypergraph $\mathcal{H} \subset \G_k[\mathcal{F}]$ satisfying the conditions of the lemma. The degree function of $\mathcal{H}$ will simply be denoted by $d$, for simplicity.

  We say that a non-empty family $\mathcal{A} \subset \mathcal{P}(n)$ is \textit{saturated} if $d(\mathcal{A}) = \lfloor (\delta m)^{k - |\mathcal{A}|} \rfloor$, that is, if no edge of $\mathcal{F}$ containing this family can be added to the hypergraph $\mathcal{H}$ without violating condition~$(b)$. A family $\mathcal{B} \subset \mathcal{P}(n)$ is \textit{bad} if it contains a saturated sub-family $\A \subset \B$. Otherwise we say that $\B$ is \textit{good}. With this terminology, the conclusion of Lemma~\ref{lemma:supersaturation} is that $\G_k[\mathcal{F}]$ contains a good edge. Indeed, since a good edge $f \in \G_k[\F]$ is not saturated, then $d(f) < 1$, and so $f \not\in \HH$.

   The following easy lemma will be a crucial tool in the proof of Lemma~\ref{lemma:supersaturation}. It says that there are not too many ways to turn a good family bad.

  \begin{lemma}\label{lemma:bad_density_penalty}
    For any good $\mathcal{A} \subset \mathcal{P}(n)$, there are at most $2^{|\mathcal{A}|} \cdot 2 \delta k m$ sets $F \in \mathcal{P}(n)$ for which $\{F\} \cup \mathcal{A}$ is bad and $\{F\}$ is not saturated.
  \end{lemma}

  \begin{proof}
    The result follows from a simple double-counting argument, which we spell out below. Since $\mathcal{A}$ is good, any saturated sub-family of $\{F\} \cup \mathcal{A}$ must contain $F$. In other words, any $F$ such that $\{F\} \cup \mathcal{A}$ is bad belongs to
    \begin{equation*}
      S(\mathcal{B}) := \big\{ F \in \mathcal{P}(n) \,:\, d\big( \{F\} \cup \mathcal{B} \big) = \big\lfloor (\delta m)^{k - |\mathcal{B}| - 1} \big\rfloor \big\}
    \end{equation*}
    for some $\mathcal{B} \subset \mathcal{A}$. Moreover, if $\{F\}$ is not saturated, then $\mathcal{B}$ cannot be empty. Therefore, it is enough to bound the size of $S(\mathcal{B})$ when $\mathcal{B}$ is non-empty. We do so by noting that
    \begin{equation*}
      |S(\mathcal{B})| \big\lfloor (\delta m)^{k-|\mathcal{B}|-1} \big\rfloor = \sum_{F \in S(\mathcal{B})} d\big( \{F\} \cup \mathcal{B} \big) \, \le \, kd(\mathcal{B}) \, \le \, k (\delta m)^{k - |\mathcal{B}|},
    \end{equation*}
    where the first inequality is true because each edge of $\mathcal{H}$ containing $\mathcal{B}$ contributes at most $k$ to the sum. Since $m \ge \delta^{-1}$, we obtain $|S(\mathcal{B})| \le 2\delta k m$. The claimed bound now follows by summing over all choices of $\mathcal{B}$.
  \end{proof}

  Similarly, noting that $S(\emptyset) = \big\{ F \in \P(n) : \{F\} \textup{ is saturated} \big\}$, we have
  $$|S(\emptyset)| \lfloor (\delta m)^{k-1} \rfloor \le k \cdot e(\mathcal{H}).$$
  By condition $(a)$ and the bound $m \ge \delta^{-1}$, it follows that $|S(\emptyset)| \le 2\delta k \binom{n}{n/2}$. Thus, by adjusting $\alpha$ slightly if necessary, we can remove the elements of $S$ from $\F$. Therefore, from now on we will assume that $\mathcal{F}$ contains no saturated sets.

  We will next sketch the proof of Lemma~\ref{lemma:supersaturation}. The key idea is that if we choose $F_1$ to be of \textit{minimal cardinality} such that the ``density'' of $k$-chains below $F_1$ (see Definition~\ref{defn:chain_density}) is bigger than $\alpha/k$ (see Lemma~\ref{lemma:big_max_density}), then only few of those $k$-chains will be bad, and hence at least one of them will be good. In order to bound the density of bad $k$-chains below $F_1$, let us define a chain $F_1 \supsetneq \cdots \supsetneq F_\ell$ to be $\textit{critical}$ if $\{F_1, \ldots, F_{\ell-1}\}$ is good but $\{F_1, \ldots, F_{\ell}\}$ is not. We will use Lemma~\ref{lemma:bad_density_penalty} to show that the density of critical $\ell$-chains is small (see Lemma~\ref{lemma:critical_chains}). We will then use the minimality of $F_1$ to deduce that the operation of extending critical $\ell$-chains to bad $k$-chains only increases the density by a bounded factor.

In order to make the above sketch more precise, let us next formalize the notion of density that we will use. This definition is inspired by the work of Das, Gan and Sudakov~\cite{2013arXiv1302.5210D}, see Lemma~\ref{lemma:dgs} below. We remark that, despite its name, the $\ell$-chain density of a set is not bounded above by 1, and in fact can be as large as $\Omega(n^{\ell-1})$.

  \begin{defn}\label{defn:chain_density}
    The $\ell$-chain density of a set $F_1 \in \mathcal{F}$, denoted by $\c_\ell(F_1)$, is given by
    \begin{equation*}
      \c_\ell(F_1) := \sum_{\substack{F_2, \ldots, F_\ell \in \mathcal{F} \\ F_1 \supsetneq F_2 \supsetneq \cdots \supsetneq F_\ell}} \binom{|F_1|}{|F_2|}^{-1}\cdots\binom{|F_{\ell-1}|}{|F_\ell|}^{-1}
    \end{equation*}
    In particular, $\c_1(F) = 1$ for all $F \in \mathcal{F}$.
  \end{defn}

    The following lemma is essentially due to Das, Gan and Sudakov~\cite{2013arXiv1302.5210D}. Since it was not explicitly stated in their paper, we will give the proof for completeness.

  \begin{lemma}[Das, Gan and Sudakov]\label{lemma:dgs}
    For any fixed $1 \leq i < j \le k$, we have
    \begin{equation*}
      \sum_{F \in \mathcal{F}} \frac{1}{\binom{n}{|F|}} \big( \c_i(F) - \c_j(F) \big) \le \max_{s \in \mathbb{N}} \binom{s}{i} - \binom{s}{j}.
    \end{equation*}
  \end{lemma}

  \begin{proof}
    Following the permutation method, say a permutation $\pi$ of $[n]$ \textit{contains} a set $F$ if $F = \{\pi(1), \ldots, \pi(|F|)\}$. Moreover, say it contains a chain if it contains all sets of the chain. Note that the number of permutations containing a given chain $F_1 \supsetneq \cdots \supsetneq F_\ell$ is
    \begin{equation*}
      (n - |F_1|)! \cdot |F_1 \setminus F_2|! \cdots |F_{\ell-1} \setminus F_\ell|! \cdot |F_\ell|! = n! \cdot \binom{n}{|F_1|}^{-1} \binom{|F_1|}{|F_2|}^{-1} \ldots \binom{|F_{\ell-1}|}{|F_\ell|}^{-1},
    \end{equation*}
    and so, denoting by $X_\ell(\pi)$ the number of $\ell$-chains contained in $\pi$, the expected value of $X_\ell$ with respect to the uniform probability measure on the set of permutations is
    \begin{equation*}
      \mathbb{E}(X_\ell) = \sum_{\substack{F_1, \ldots, F_\ell \in \mathcal{F} \\ F_1 \supsetneq \ldots \supsetneq F_\ell}} \binom{n}{|F_1|}^{-1} \binom{|F_1|}{|F_2|}^{-1} \ldots \binom{|F_{\ell-1}|}{|F_\ell|}^{-1} = \sum_{F_1 \in \mathcal{F}} \left. \c_\ell(F_1) \middle/ \binom{n}{|F_1|}\right. .
    \end{equation*}
    On the other hand, since sets contained in a single permutation always form a chain, $X_\ell(\pi)$ equals $\binom{s}{\ell}$, where $s$ is the number of elements of $\mathcal{F}$ contained in $\pi$. We deduce that
    \begin{equation*}
      X_i(\pi) - X_j(\pi) \le \max_{s \in \mathbb{N}} \binom{s}{i} - \binom{s}{j},
    \end{equation*}
    and the conclusion follows by taking the expected value of both sides.
  \end{proof}

  A very useful feature of Lemma~\ref{lemma:dgs} is that the upper bound it provides does not depend on $n$. We will next use this to show that $\ell$-chain densities cannot decrease too quickly as a function of $\ell$, and hence that it is enough to upper bound the $k$-chain density of a set whenever we want an upper bound for all of its lower densities.


  \begin{lemma}\label{lemma:density_almost_monotone}
    For every $F \in \mathcal{F}$ and $1 \leq \ell < k$, we have $\c_\ell(F) \le \c_k(F) + 4^k$.
  \end{lemma}

  \begin{proof}
    The result is trivial for $\ell = 1$, as $\c_1(F) = 1$. For $\ell \ge 2$, we can use the identity
    \begin{equation*}
      \c_\ell(F) = \sum_{\substack{F_2 \in \mathcal{F} \\ F \supsetneq F_2}} \binom{|F|}{|F_2|}^{-1} \sum_{\substack{F_3,\ldots,F_k \in \mathcal{F} \\ F_2 \supsetneq \cdots \supsetneq F_\ell}} \binom{|F_2|}{|F_3|}^{-1} \cdots \binom{|F_{\ell-1}|}{|F_\ell|}^{-1} = \sum_{F \supsetneq F_2 \in \mathcal{F}} \left. \c_{\ell-1}(F_2) \middle/ \binom{|F|}{|F_2|} \right.
    \end{equation*}
    together with Lemma~\ref{lemma:dgs} (applied to the hypercube of subsets of $F$) to obtain
    \begin{equation*}
      \c_\ell(F) - \c_k(F) = \sum_{F \supsetneq F_2 \in \mathcal{F}} \frac{1}{\binom{|F|}{|F_2|}} (\c_{\ell-1}(F_2) - \c_{k-1}(F_2)) \le \max_{s \in \mathbb{N}} \binom{s}{\ell-1} - \binom{s}{k-1}.
    \end{equation*}
    Since the function being maximized is negative for all $s \ge 2k-1$, the right side is at most $\binom{2k-1}{\ell-1} \le 4^k$, which proves the result.
  \end{proof}

  Lemma~\ref{lemma:dgs} also allows us to deduce that at least one element of our family has large $k$-chain density, as we show in the following pigeonhole-like observation.

  \begin{lemma}\label{lemma:big_max_density}
    If $0 \leq \alpha \leq 1$ and $|\mathcal{F}| \ge (k-1+\alpha) \binom{n}{n/2}$, then $\max_F \c_k(F) \ge \alpha/k$.
  \end{lemma}

  \begin{proof}
    By Lemma~\ref{lemma:dgs} with $i = 1$ and $j = k$, and since $c_1(F) = 1$, we have
    \begin{equation*}
      \sum_{F \in \mathcal{F}} \frac{1}{\binom{n}{|F|}} (1 - \c_k(F)) \le \max_{s \in \mathbb{N}} \binom{s}{1} - \binom{s}{k} = k - 1.
    \end{equation*}
    However, if the desired conclusion were not true, we would have
    \begin{equation*}
      \sum_{F \in \mathcal{F}} \frac{1}{\binom{n}{|F|}} \left( 1 - \c_k(F) \right) > \sum_{F \in \mathcal{F}} \frac{1}{\binom{n}{n/2}} \left( 1 - \frac{\alpha}{k} \right) \ge (k-1+\alpha) \cdot \frac{k-\alpha}{k} \ge k-1,
    \end{equation*}
    where, for the last step, note that equality holds when $\alpha \in \{0,1\}$.
  \end{proof}

    Finally, we will need the following lemma, which bounds the density of critical $\ell$-chains. It is a simple consequence of Lemma~\ref{lemma:bad_density_penalty} and our assumption that $m \le \binom{|F|}{|G|}$ for every $F,G \in \F$ with $F \supsetneq G$.

  \begin{lemma}\label{lemma:critical_chains}
   For every $F_1 \in \mathcal{F}$ and $1 \le \ell < k$,
    \begin{equation}\label{eq:critical}
      \sum_{\substack{F_2, \ldots, F_{\ell+1} \in \mathcal{F} \\ F_1 \supsetneq \cdots \supsetneq F_{\ell+1} \text{ critical}}} \binom{|F_1|}{|F_2|}^{-1} \cdots \binom{|F_{\ell}|}{|F_{\ell+1}|}^{-1} \leq 2^{\ell} \cdot 2\delta k \cdot \c_{\ell}(F_1)
    \end{equation}
  \end{lemma}

  \begin{proof}
 Recall that if $F_1 \supsetneq \cdots \supsetneq F_{\ell+1}$ is critical, then $\{F_1, \ldots, F_\ell\}$ is good but $\{F_1, \ldots, F_{\ell+1}\}$ is not. By Lemma~\ref{lemma:bad_density_penalty}, it follows that the left-hand side of~\eqref{eq:critical} is at most
    \begin{equation*}
      \sum_{\substack{F_2, \ldots, F_{\ell} \in \mathcal{F} \\ F_1 \supsetneq \cdots \supsetneq F_{\ell}}} \binom{|F_1|}{|F_2|}^{-1} \cdots \binom{|F_{\ell-1}|}{|F_{\ell}|}^{-1} \cdot 2^{\ell} \cdot 2\delta k m \cdot \max_{F_{\ell} \supsetneq F_{\ell+1} \in \mathcal{F}} \binom{|F_{\ell}|}{|F_{\ell+1}|}^{-1}.
    \end{equation*}
    The result then follows from our upper bound on $m$ and the definition of $c_\ell(F_1)$.
  \end{proof}

  We are now ready to carry out the plan outlined above, and prove Lemma~\ref{lemma:supersaturation}.

  \begin{proof}[Proof of Lemma~\ref{lemma:supersaturation}]
  We may assume, without loss of generality, that $0 < \alpha < 1$. Let $F_1$ be of minimal cardinality such that $\c_k(F_1) \ge \alpha/k$ (note that at least one such $F_1$ exists, by Lemma~\ref{lemma:big_max_density}). We claim that
    \begin{equation}\label{eq:half_are_good}
      \sum_{\substack{F_2, \ldots, F_k \in \mathcal{F} \\ F_1 \supsetneq \cdots \supsetneq F_{k} \text{ bad}}} \binom{|F_1|}{|F_2|}^{-1} \cdots \binom{|F_{k-1}|}{|F_{k}|}^{-1} \leq \frac{\c_k(F_1)}{2},
    \end{equation}
    which immediately implies that the total $k$-chain density of \emph{good} chains is positive, and therefore that at least one good chain exists. In order to prove \eqref{eq:half_are_good}, notice that every bad $k$-chain $F_1 \supsetneq \cdots \supsetneq F_k$ is associated with a unique $1 \le \ell < k$ such that $F_1 \supsetneq \cdots \supsetneq F_{\ell+1}$ is critical. As such, we can write the left side of \eqref{eq:half_are_good} as
    \begin{equation*}
      \sum_{\ell = 1}^{k-1} \Bigg( \sum_{\substack{F_2, \ldots, F_{\ell+1} \in \mathcal{F} \\ F_1 \supsetneq \cdots \supsetneq F_{\ell+1} \text{ critical}}} \binom{|F_1|}{|F_2|}^{-1} \cdots \binom{|F_{\ell}|}{|F_{\ell+1}|}^{-1} \cdot \c_{k-\ell}(F_{\ell+1}) \Bigg).
    \end{equation*}
    We will proceed by bounding each term of the outer sum separately, so fix $1 \le \ell < k$. By Lemma~\ref{lemma:density_almost_monotone} and the minimality of $F_1$, we have $\c_{k-\ell}(F_{\ell+1}) \leq \c_k(F_{\ell+1}) + 4^k < \alpha/k + 4^k < 5^k$. Using this bound and Lemma~\ref{lemma:critical_chains}, we obtain
    \begin{equation}\label{eq:bound_for_each_ell}
      \sum_{\substack{F_2, \ldots, F_{\ell+1} \in \mathcal{F} \\ F_1 \supsetneq \cdots \supsetneq F_{\ell+1} \text{ critical}}} \binom{|F_1|}{|F_2|}^{-1} \cdots \binom{|F_{\ell}|}{|F_{\ell+1}|}^{-1} \cdot \c_{k-\ell}(F_{\ell+1}) \leq 2^{\ell} \cdot 2\delta k \cdot \c_\ell(F_1) \cdot 5^k.
    \end{equation}
    Using Lemma~\ref{lemma:density_almost_monotone} once again for the bound $\c_\ell(F_1) \le \c_k(F_1) + 4^k$ and summing \eqref{eq:bound_for_each_ell} over $1 \le \ell < k$, we conclude that
    \begin{equation*}
      \sum_{\substack{F_2, \ldots, F_k \in \mathcal{F} \\ F_1 \supsetneq \cdots \supsetneq F_{k} \text{ bad}}} \binom{|F_1|}{|F_2|}^{-1} \cdots \binom{|F_{k-1}|}{|F_{k}|}^{-1} = \delta \cdot 2^{O(k)} \cdot ( \c_k(F_1) + 4^k ) = \frac{\delta \cdot 2^{O(k)}}{\alpha} \cdot \c_k(F_1),
    \end{equation*}
since $\c_k(F_1) \ge \alpha / k$. The right side can be made less than $\c_k(F_1)/2$ by choosing $\delta$ to be small (only as a function of $\alpha$ and $k$), and so the proof is complete.
  \end{proof}

  \section{Proof of Theorem~\ref{thm:main}}\label{sec:putting_it_together}

    In this section we will deduce Theorem~\ref{thm:main} from the results of the previous two sections. More precisely, we will use Corollary~\ref{cor:specialized_hcl}
    to prove a `fingerprint theorem' (Theorem~\ref{thm:container_application}, below), which easily implies Theorem~\ref{thm:main}. A \emph{coloured vertex set} is simply a family $\A \subset \mathcal{P}(n)$ together with a function $c \colon \A \to \mathbb{N}$. Recall that $\G_k$ denotes the $k$-uniform hypergraph whose edges encode $k$-chains. We will need the following definition.

    \begin{defn}\label{def:fingerprints}
      A fingerprint of $\mathcal{G}_k$ is a family $\mathcal{S}$ of coloured vertex sets, together with:
      \begin{enumerate}
        \item[$(a)$] A fingerprint function $T \colon \mathcal{I}(\mathcal{G}_k) \to \mathcal{S}$ with $T(I) \subset I$ for every $I \in \mathcal{I}(\mathcal{G}_k)$.
        \item[$(b)$] A container function $C \colon \mathcal{S} \to \mathcal{P}(V(\mathcal{G}_k))$ such that $I \subset C(T(I))$ for every $I \in \mathcal{I}(\mathcal{G}_k)$.
      \end{enumerate}
    \end{defn}

    Each $S \in \mathcal{S}$ should be thought of as a sequence of subsets of $V(\HH)$ given by repeated application of the Hypergraph Container Lemma. The container function is obtained by applying the sequence of functions $f$ given by these repeated applications. We will prove the following theorem.

    \begin{thm}\label{thm:container_application}
   For every $k \ge 2$ and $\eps > 0$, there exist a constant $K = K(\varepsilon, k) > 0$ and a fingerprint $(\mathcal{S}, T, C)$ of $\G_k$ such that the following hold:
      \begin{enumerate}
        \item[$(a)$] Every $S \in \mathcal{S}$ satisfies $|S| \leq \frac{K}{n}\binomn$; \smallskip
        \item[$(b)$] The number of members of $\mathcal{S}$ of size $s$ is at most \[\bigg(\frac{K\binomn}{s}\bigg)^s \exp\left(\frac{K}{n}\binomn\right);\]
        \item[$(c)$] $|C(T(I))| \leq (k-1+\varepsilon)\binomn$ for every $I \in \mathcal{I}(\G_k)$.
      \end{enumerate}
    \end{thm}

    Before proving Theorem~\ref{thm:container_application}, let us see how it implies Theorem~\ref{thm:main}.
    \begin{proof}[Proof of Theorem~\ref{thm:main}]
      Let $k \ge 2$ and $\eps > 0$ be arbitrary, and let $K = K(\eps,k) > 0$ and $(\mathcal{S}, T, C)$ be the constant and fingerprint given by Theorem~\ref{thm:container_application}. Let $n \in \N$ be sufficiently large, and note that $pn \geq K\varepsilon^{-1}$, since $pn \to \infty$. If $I \subset \mathcal{P}(n, p)$ is an independent set of $\mathcal{G}_k$ of size at least $(k-1+3\varepsilon)p\binomn$, then it follows that $T(I) \subset \mathcal{P}(n, p)$ and
      $$ \big| C(T(I)) \cap \mathcal{P}(n, p) \big|  \ge \big( k - 1 + 3\varepsilon \big) p \binomn.$$
      Let $X$ be the number of elements of $\mathcal{S}$ for which these two properties hold. Then
      \begin{equation*}
        \mathbb{E}(X) \leq \sum_{A \in \mathcal{S}} \mathbb{P}\big(A \subset \mathcal{P}(n, p)\big) \cdot \mathbb{P}\left(\big|(C(A) \setminus A) \cap \mathcal{P}(n, p)\big| \geq (k-1+2\varepsilon)p\binomn\right),
      \end{equation*}
      where we used that $|A| \leq \varepsilon p\binomn$ by the lower bound on $pn$ and Theorem~\ref{thm:container_application}~$(a)$. Hence, by the properties of $(\mathcal{S}, T, C)$ guaranteed by Theorem~\ref{thm:container_application}, and Chernoff's inequality,
      \begin{align*}
        \mathbb{E}(X) &\leq \sum_{s=1}^{\frac{K}{n}\binomn} \bigg(\frac{K \binomn}{s}\bigg)^s \exp\left(\frac{K}{n}\binomn\right) \cdot p^s \cdot \exp\left(- \frac{\varepsilon^2 p}{3} \binomn\right) \\
        &\leq \frac{K}{n} \binomn \exp\left( \frac{K \log(pn)}{n} \binomn + \frac{K}{n}\binomn - \frac{\varepsilon^2 p}{3} \binomn \right),
      \end{align*}
      since the summand is increasing in $s$ on the interval $\big($$0, (Kp/e) \binomn\big)$, and $K/n \ll Kp/e$. Therefore, by Markov's inequality, and since $pn \gg \log(pn) \gg 1$, we have
      \begin{equation*}
        \mathbb{P}\bigg(\alpha\big(\mathcal{P}(n, p)\big) \geq \big( k - 1 + 3\varepsilon \big)p\binomn\bigg) \leq \exp\left(-\frac{\varepsilon^2p}{6}\binomn\right) \to 0
      \end{equation*}
      as $n \to \infty$, as required.
    \end{proof}

    It only remains to prove Theorem~\ref{thm:container_application}. We will use a straightforward but technical lemma.

    \begin{lemma}\label{lemma:slogs}
      Let $M > 0$, $s > 0$ and $0 < \delta < 1$. For any finite sequence $(a_1,\ldots,a_m)$ of real numbers summing to $s$ such that $1 \leq a_j \leq (1-\delta)^j M$ for each $j \in [m]$, we have
      \begin{equation*}
        s \log s \leq \sum_{j=1}^m a_j \log a_j + O(M).
      \end{equation*}
    \end{lemma}

    We remark that this lemma is especially easy to prove if $m = O(1)$, which will be the case in our application. However, it is not much harder to prove in general, and this more general version is necessary for other applications, cf.~\cite[Section~6]{2013arXiv1309.2927M}.

    \begin{proof}[Proof of Lemma~\ref{lemma:slogs}]
      Fix $m \in \mathbb{N}$ and note that, by compactness, we can assume that the sequence $(a_1, \ldots, a_m)$ achieves the minimum of $\sum_{j=1}^m x_j \log x_j$ subject to the given conditions. Let
      \begin{equation*}
        J_1 = \{j \in [m] : a_j < (1-\delta)^j M\}
      \end{equation*}
      and $J_2 = [m] \setminus J_1$; define also $s_i = \sum_{j \in J_i} a_j$ for $i \in \{1,2\}$. The convexity of $x \log x$ implies that all of the elements of the subsequence $(a_j)_{j \in J_1}$ are equal and that $J_1 = [t]$ for some $t \in \{0, \ldots, m\}$, so that $s_1 \leq t (1-\delta)^t M$. Note that $s = \sum_j a_j = O(M)$ and
      \begin{align*}
        s_2 \log M - \sum_{j \in J_2} a_j \log a_j = \sum_{j \in J_2} a_j \log \frac{M}{a_j} \leq \sum_{j=1}^\infty (1-\delta)^j M \log \frac{1}{(1-\delta)^j} = O(M).
      \end{align*}
      We are done if $t = 0$, so assume $t \geq 1$. By convexity, $s \log s \leq s_1 \log s_1 + s_2 \log s_2 + s \log 2$. Hence, recalling that $a_1 = \ldots = a_t = s_1/t$, we have
      \begin{align*}
        s \log s &\leq s_1 \log \frac{s_1}{t} + s_1 \log t + s_2 \log s_2 + O(M) \\
                 &\leq \sum_{j \in J_1} a_j \log a_j + t (1-\delta)^t M \log t + \sum_{j \in J_2} a_j \log a_j + O(M) \\
                 &= \sum_{j=1}^m a_j \log a_j + O(M),
      \end{align*}
    as claimed.
    \end{proof}

  We are now ready to prove the `fingerprint theorem', and thus complete the proof of Theorem~\ref{thm:main}.

    \begin{proof}[Proof of Theorem~\ref{thm:container_application}]
Let $k \ge 2$ and $\eps > 0$ be arbitrary, let $\delta = \delta(\varepsilon,k) > 0$ be given by Corollary~\ref{cor:specialized_hcl}, choose a large constant $K = K(\varepsilon, k,\delta)$, and let $n \in \mathbb{N}$ be sufficiently large. For a given $I \in \mathcal{I}(\mathcal{G}_k)$, we will apply Corollary~\ref{cor:specialized_hcl} a certain number of times, which we will denote by $m = m(I)$, to construct two sequences $C_1, \ldots, C_{m+1}$ and $T_1, \ldots, T_m$ of subsets of $V(\G_k)$. The construction will inductively maintain the following properties:
      \begin{enumerate}
         \item[$(i)$] $I \subset C_{i+1} \cup T_1 \cup \cdots \cup T_i$,
         \item[$(ii)$] The sets $C_{i+1}, T_1, \ldots, T_i$ are pairwise disjoint,
         \item[$(iii)$] $C_{i+1}$ only depends on $C_i$ and $T_i$,
         \item[$(iv)$] $|C_{i+1}| \le (1-\delta) |C_i|$.
      \end{enumerate}
      To do this, first set $C_1 := \mathcal{P}(n)$. As long as $|C_i| \geq (k-1+\varepsilon) \binomn$, let $T_i \subset I \cap C_i$ and $f_i$ be given by Corollary~\ref{cor:specialized_hcl} applied to $C_i$, and set $C_{i+1} := f_i(T_i) \setminus T_i \subset C_i \setminus T_i$. We stop when we can no longer apply Corollary~\ref{cor:specialized_hcl}, that is, when $|C_{m+1}| < (k-1+\varepsilon)\binomn$.

      We define our fingerprint $(\mathcal{S}, T, C)$ of $\G_k$ by setting
      \begin{equation*}
        T(I) := (T_1, \ldots, T_m) \qquad \text{and} \qquad C(T(I)) := C_{m+1} \cup T_1 \cup \cdots \cup T_m,
      \end{equation*}
      and letting $\mathcal{S} := \{T(I) : I \in \mathcal{I}(\mathcal{G}_k)\}$. Note that Property~$(iii)$ implies $C$ is well-defined, while Property~$(i)$ guarantees that it is a container function.

      In order to check that the constructed fingerprint satisfies the conditions of the theorem, we first bound the sizes of the fingerprints and the number of iterations of the above procedure. To do so, let $2 \le m_0 \le m$ be minimal such that $|C_{m_0}| \le 3k \binomn$, and observe that, by Property~$(iv)$ and the definition~\eqref{def:tau} of $\tau(A)$,
      \begin{equation}\label{eq:fingerprint_upper_bound}
        \tau(C_i) |C_i| \le \begin{cases} n^{-3} \cdot 2^n &\text{if } i < m_0, \\
                                             n^{-1} \cdot (1-\delta)^{i-m_0} \cdot 3k \binomn & \text{otherwise.}
        \end{cases}
      \end{equation}
 The geometric decay of $|C_i|$ moreover immediately implies that $m = O(\log n)$. We thus obtain
      \begin{equation}\label{eq:sum_fingerprints}
        \sum_{i=1}^{m_0-1} \tau(C_i)|C_i| \le \frac{m \cdot 2^n}{n^3} \ll \frac{1}{n^2}\binomn \qquad\text{and}\qquad \sum_{i = m_0}^m \tau(C_i)|C_i| = \frac{O(1)}{n} \binomn.
      \end{equation}
      Since $|T(I)| = \sum_{i=1}^m |T_i| \le \sum_{i=1}^m k\tau(C_i) |C_i|$, adding the two bounds immediately proves $(a)$. Also, since $n$ is sufficiently large,
      \begin{equation*}
        |C(T(I))| = |C_{m+1}| + |T_1 \cup \cdots \cup T_m| \le (k-1+2\varepsilon)\binomn,
      \end{equation*}
      which proves $(c)$, since $\eps > 0$ was arbitrary.

      It only remains to prove $(b)$, which follows using Lemma~\ref{lemma:slogs}. The first step is to partition the collection of $s$-sets in $\mathcal{S}$ into subfamilies $\mathcal{S}(\hat{m}_0, \mathbf{t})$, where for given $\hat{m}_0 \in \N$ and $\mathbf{t} = (t_1,\ldots,t_{\hat{m}}) \in \N^{\hat{m}}$, we define $\mathcal{S}(\hat{m}_0, \mathbf{t})$ to be set of all $(T_1, \ldots, T_{\hat{m}}) \in \mathcal{S}$ such that $\hat{m}_0$ is the smallest integer for which $|C_{\hat{m}_0}| \le 3k \binomn$ and moreover $|T_i| = t_i$ for each $i \in [\hat{m}]$.

   In order to bound the number of elements of $\mathcal{S}(\hat{m}_0, \mathbf{t})$ of size $s$, set $s_1 = \sum_{i=\hat{m}_0}^{\hat{m}} t_i$, and observe that
 \begin{equation}\label{eq:log_binomial_bound}
          \sum_{i=\hat{m}_0}^{\hat{m}} t_i \log \frac{1}{t_i} \leq s_1 \log \frac{1}{s_1} + \frac{O(1)}{n} \binomn,
        \end{equation}
   by Lemma~\ref{lemma:slogs} and the second bound in~\eqref{eq:fingerprint_upper_bound}. Since each $T_i$ is a subset of the corresponding $C_i$, we can use the trivial bound $|C_i| \le 2^n$ and the definition of ${\hat{m}_0}$ to write
        \begin{align*}
        \big|\mathcal{S}({\hat{m}_0}, \mathbf{t})\big|&\leq \prod_{i=1}^{{\hat{m}_0}-1} \binom{2^n}{t_i} \prod_{i={\hat{m}_0}}^{\hat{m}} \binom{3k\binomn}{t_i} \\
                                        &\leq \left( \prod_{i=1}^{{\hat{m}_0}-1} 2^{t_in} \right) \left(\left[3ek \cdot \binomn\right]^{s_1} \prod_{i={\hat{m}_0}}^{\hat{m}} \left(\frac{1}{t_i}\right)^{t_i}\right) \\
                                        &\leq \left(\frac{K \binomn}{s_1}\right)^{s_1} \exp\left(\frac{K}{n}\binomn\right)
        \end{align*}
        where the final step follows from the first sum in \eqref{eq:sum_fingerprints} and from applying the exponential function to \eqref{eq:log_binomial_bound}. Finally, note that the right-hand side is monotone in $s_1$ on the interval $\big(0, K\binomn/e\big)$,
 and we can therefore replace $s_1$ by $s$. Summing over the (at most
 $n^{O(n)}$) choices of $\mathbf{t}$, ${\hat{m}_0}$ and $\hat{m}$, the claimed bound follows.
    \end{proof}


\begin{thebibliography}{1}

\bibitem{2012arXiv1204.6530B}
J.~{Balogh}, R.~{Morris}, and W.~{Samotij},
\newblock {Independent sets in hypergraphs},
\newblock \emph{J. Amer. Math. Soc.}, \textbf{28} (2015), 669--709.

\bibitem{BMT}
J.~{Balogh}, R.~{Mycroft} and A.~{Treglown},
\newblock{A random version of Sperner's theorem},
\newblock \emph{J. Combin. Theory, Ser. A}, \textbf{128} (2014), 104--110.

\bibitem{MR0183653}
B.~Bollob\'as,
\newblock On generalized graphs,
\newblock {\em Acta Math. Acad. Sci. Hungar}, \textbf{16} (1965), 447--452.

\bibitem{2010arXiv1011.4310C}
D.~{Conlon} and W.~T. {Gowers},
\newblock {Combinatorial theorems in sparse random sets},
\newblock submitted.

\bibitem{2013arXiv1302.5210D}
S.~{Das}, W.~{Gan}, and B.~{Sudakov},
\newblock {Sperner's Theorem and a Problem of Erd\H{o}s--Katona--Kleitman},
\newblock \emph{Combin. Prob. Computing}, \textbf{24} (2015), 585--608.

\bibitem{Erdos}
P.~{Erd\H{o}s},
\newblock On a Lemma of Littlewood and Offord,
\newblock {\em Bull. Amer. Math. Soc.}, \textbf{51} (1945), 898--902.

\bibitem{MR0232685}
D.~{Kleitman},
\newblock A conjecture of {E}rd{\H o}s-{K}atona on commensurable pairs among
  subsets of an {$n$}-set,
\newblock In {\em Theory of {G}raphs ({P}roc. {C}olloq., {T}ihany, 1966)},
  pages 215--218. Academic Press, New York, 1968.

\bibitem{KK} Y. Kohayakawa and B. Kreuter, The Width of Random Subsets of Boolean Lattices, \emph{J. Combin. Theory, Ser. A}, \textbf{100} (2002), 376--386.

\bibitem{KKO} Y. Kohayakawa, B. Kreuter and D.~Osthus, The length of random subsets of Boolean lattices, \emph{Random Struct. Algorithms}, \textbf{16} (2000), 177--194.

\bibitem{K} B.~Kreuter, Small sublattices in random subsets of Boolean lattices, \emph{Random Struct. Algorithms}, \textbf{13} (1998), 383--407.

\bibitem{MR0194348}
D.~Lubell,
\newblock A short proof of {S}perner's lemma,
\newblock {\em J. Combin. Theory}, \textbf{1} (1966), 299.

\bibitem{MR0150049}
L.D.~Me{\v{s}}alkin,
\newblock A generalization of {S}perner's theorem on the number of subsets of a finite set,
\newblock {\em Teor. Verojatnost. i Primenen}, \textbf{8} (1963), 219--220.

\bibitem{2013arXiv1309.2927M}
R.~{Morris} and D.~{Saxton},
\newblock {The number of $C_{2\ell}$-free graphs},
\newblock to appear in \emph{Adv. Math.}

\bibitem{MR1757282}
D.~{Osthus},
\newblock Maximum antichains in random subsets of a finite set,
\newblock {\em J. Combin. Theory, Ser. A}, \textbf{90} (2000), 336--346.

\bibitem{MR0144366}
A.~{R{\'e}nyi},
\newblock On random subsets of a finite set,
\newblock {\em Mathematica (Cluj)}, 355--362, 1961.

\bibitem{2012arXiv1204.6595S}
D.~{Saxton} and A.~{Thomason},
\newblock {Hypergraph containers},
\newblock \emph{Inventiones Math.}, \textbf{201} (2015), 925--992.

\bibitem{Schacht09}
M.~{Schacht},
\newblock {Extremal results for random discrete structures},
\newblock submitted.

\bibitem{MR1544925}
E.~{Sperner},
\newblock Ein {S}atz \"uber {U}ntermengen einer endlichen {M}enge,
\newblock {\em Math. Z.}, \textbf{27} (1928), 544--548.

\bibitem{MR0067086}
K.~Yamamoto,
\newblock Logarithmic order of free distributive lattice,
\newblock {\em J. Math. Soc. Japan}, \textbf{6} (1954), 343--353.

\end{thebibliography}
\end{document}